\documentclass[a4,12pt]{article}
\usepackage{amsmath, amssymb, amsfonts, amstext, amsthm, textcomp}
\usepackage[mathscr]{euscript}
\newtheorem{thm}{Theorem}[section]

\newtheorem{defn}{Definition}[section]
\newtheorem{rem}{Remark}[section]
\newtheorem{ex}{Example}[section]
\begin{document}

\noindent
\begin{center}
 {\bf \large Vanishing Pseudo Schur Complements, Reverse Order Laws, Absorption Laws and Inheritance Properties}
\vspace{.5cm}

             { \bf Kavita Bisht}\\
Department of Mathematics\\
                         Indian Institute of Technology Madras\\

                               Chennai 600 036, India \\
          
                and\\
              {\bf K.C. Sivakumar} \\
                         Department of Mathematics\\
                         Indian Institute of Technology Madras\\

                               Chennai 600 036, India. \\
\end{center}
\begin{center}
{\bf Abstract}
\end{center}
The problem of when the vanishing of a (generalized) Schur complement of a block matrix (corresponding to the leading principal subblock) implies that the other (generalized) Schur complement (corresponding to the trailing principal subblock) is zero, is revisited. A simple proof is presented. Absorption laws for two important classes of generalized inverses are considered next. Inheritance properties of the generalized Schur compements in relation to the absorption laws are derived. Inheritance by the generalized principal pivot transform is also studied. 

{\bf Keywords:} Moore-Penrose inverse, Group inverse, Pseudo principal pivot transform, Pseudo Schur complements, Absorption law, Inheritance properties. 

{\bf AMS Subject Classification:} 15A09, 15A24

\newpage
\section{Introduction}
Let $M$ be a block matrix with real entries, partitioned as 
\begin{center}
 $\left( \begin{array}{lr} A  & B \\ C &  D  \end{array}
\right)$
\end{center}
where $A$ is nonsingular. Then the Schur complement of $A$ in $M$ denoted by $M/A$ is given by $D-CA^{-1}B$. This notion has proved to be a fundamental object in many applications like numerical analysis, statistics and operator inequalities, to name a few. The notion of Schur complement was extended to the case, where $A^{-1}$ was replaced by the Moore-Penrose inverse of $A$, by Albert \cite{albert1969conditions} who studied the positive definiteness and nonnegative definiteness for symmetric matrices using this formula. We also consider what was called as the complementary pseudo Schur complement in \cite{kavitaPseudoschur2015}, which extends the Schur complement of $D$ in $M$, denoted by $M/D$ and is given by $D-CA^{-1}B$. These two pseudo Schur complements of $M$ were called the associated pseudo Schur complements of $M$ \cite{wang2011associated}. In \cite{damm2009cancellation}, a question was asked as to when $M/D=0$ implies $M/A=0$. Necessary and sufficient conditions for this implication to hold were 
given in \cite{wang2011associated}. In this article, we revisit this question 
and 
give a much simpler proof of this characterization. Our proof has the advantage that it is valid for operators over Hilbert spaces. 

Let us turn to the next topic that is considered here. The absorption law in a ring $R$ says that if $a,b$ in a ring $\cal R$ with unity are invertible, then one has the trivial formula $a^{-1} + b^{-1}=a^{-1}(a+b) b^{-1}$. The absorption laws for generalized inverses of matrices were first studied in \cite{chen2008absorption} and \cite{lin2011mixedabsorption}, among others. They obtained rank conditions for the absorption law to hold for certain classes of generalized inverses. We also refer the reader to the work reported in \cite{liu2012absorption} where equivalent conditions for the absorption laws to hold, for the various generalized inverses of operators on Hilbert spaces. As the second problem, we study inheritance properties of the pseudo Schur complement in relation to the absorption laws expressed in terms of generalized inverses. 

As a third goal, we study inheritance of the absorption laws by a particular transformation of a block matrix. Let us recall this, next. Again, consider $M$, partitioned as above. If $A$ is nonsingular, then the principal pivot transform (PPT) of $M$ is the block matrix defined by 
\begin{center} 
$\begin{pmatrix}
A^{-1} & -A^{-1}B\\
CA^{-1} & D-CA^{-1}B
\end{pmatrix}$.
\end{center}
This operation of obtaining the PPT arises in many contexts, namely mathematical programming, numerical analysis and statistics, to name a few. Just as in the case of the pseudo Schur complement, it is natural to study the PPT when the usual inverses are replaced by generalized inverses. Meenakshi \cite{meenakshi1986principal}, was perhaps the first to study such a generalization for the Moore-Penrose inverse. We shall refer to this generalization as the pseudo principal pivot transform. We present results which show when the pseudo principal pivot transform inherits the absorption law property. 

We summarize the contents of the paper as follows. In the next section, we provide a brief background for the rest of the material in the article. In the third section, we study when the annihilation of one pseudo Schur compement implies that the other pseudo Schur compement is zero. In Section $4$, we consider the inheritance of the absorption law property by the two pseudo Schur complements. In the concluding section, we study inheritance by the pseudo principal pivot transform. In all the results, we first consider the case of the Moore-Penrose inverse followed by the group inverse.

\section{Preliminaries}
For $M\in \mathbb{R}^{m\times n}$, we denote the null space, the range space and the transpose of $M$ by $N(M)$, $R(M)$ and $M^T$, respectively. The {\it Moore-Penrose (generalized) inverse} of $M \in \mathbb{R}^{m\times n}$, denoted by $M^\dagger$ is the unique solution $X \in \mathbb{R}^{n\times m}$ of the equations: $M=MXM$, $X=XMX$, $(MX)^T=MX$ and $(XM)^T=XM$. Any $X$ satisfying the first two equations will be called a {\it $\{1,2\}$-inverse} of $M$. In such a case, we denote that by $X \in M\{1,2\}$. Let $M \in \mathbb{R}^{n \times n}$. If there exists $X \in \mathbb{R}^{n \times n}$ satisfying the three equations: $MXM=M$, $XMX=X$ and $MX=XM$, then such an $X$ can be shown to be unique. This unique $X$ denoted by $M^{\#}$ is called the {\it group (generalized) inverse} of $M$. Of course, if the matrix $M$ is invertible, then $M^{\dagger}=M^{\#}=M^{-1}$. It is well known that, unlike the Moore-Penrose inverse which exists for all matrices, the group inverse does not exist for all matrices. A necessary 
and sufficient condition for $M^{\#}$ to exist is the rank condition $rank(M^2)=rank(M)$. This is equivalent to the condition: $N(M^2)=N(M)$. Another equivalent statement is that the subspaces $R(M)$ and $N(M)$ are complementary. This means, in particular that if $M$ is symmetric then $M^{\#}$ exists. Moreover, for a symmetric matrix, it follows that the Moore-Penrose inverse and the group inverse coincide. Recall that $M\in \mathbb{R}^{n \times n}$ is called range-symmetric if $R(M^T)=R(M)$. The following properties of $M^{\dagger}$ are well known: $(M^{\dagger})^{T}=(M^T)^{\dagger}$; $R(M^T)=R(M^{\dagger})$; $N(M^T)=N(M^{\dagger})$.  If $x\in R(M^T)$ then $x=M^\dagger Mx$. Similar properties of $M^{\#}$ which will be used frequently in this article are: $(M^{\#})^T=(M^T)^{\#}$; $R(M^{\#})=R(M)$; $N(M^{\#})=N(M)$.  If $x \in R(M)$ then $x=M^{\#}Mx$. For $A$ and $B \in \mathbb{R}^{m\times m}$, $R(A)\subseteq R(B)$ if and only if $BB^{\dagger}A=A$ and $R(A^*)\subseteq R(B^*)$ if and only if $AB^{\dagger}B=A$. 
For matrices $A$ and $B$ of the same size with $B$ being group invertible one has $R(A)\subseteq R(B)$ if and only if $BB^{\#}A=A$ and $R(A^*)\subseteq R(B^*)$ if and only if $AB^{\#}B=A$. We refer the reader to \cite{ben2003generalized} for proofs of these statements and other details. 

As mentioned earlier, given a block matrix $M=\begin{pmatrix}
A& B\\
C & D
\end{pmatrix}$, the pseudo Schur complement of $A$ in $M$ is defined by $F=D-CA^\dagger B$ and the complementary pseudo Schur complement of $D$ in $M$ is denoted by $G=A-BD^\dagger C$. Also the pseudo principal pivot transform of $M$ relative to $A$ is defined by
\begin{center}
$H:=pppt(M,A)_{\dagger}=
\begin{pmatrix}
  A^\dagger & -A^\dagger B \\
  CA^\dagger  & F
 \end{pmatrix}$,
\end{center} where $F=D-CA^\dagger B$ \cite{meenakshi1986principal}. The complementary pseudo principal pivot transform of $M$ relative to $D$ is defined by
\begin{center}
$J:=cpppt(M,D)_{\dagger}=
\begin{pmatrix}
 G & BD^\dagger \\
-D^\dagger C & D^\dagger
\end{pmatrix}$, 
\end{center} 
where $G=A-BD^\dagger C$ \cite{kavitaPseudoschur2015}.

Next, we recall the definition of the pseudo principal pivot transform of a block matrix in terms of the group inverse. Again, for the sake of convenience, we use the same nomenclature for the group inverse as the Moore-Penrose inverse. We refer to \cite{kavi1} for some of their properties.

Let $M=\begin{pmatrix}
A & B\\
C & D
\end{pmatrix}$ such that $A$, $B$, $C$ and $D \in \mathbb{C}^{n\times n}$. Suppose $A^\#$ exists. The pseudo principal pivot transform of $M$ relative to $A$ is given by 
\begin{center}
$S=pppt(M,A)_{\#}=
\begin{pmatrix}
  A^\# & -A^\# B \\
  CA^\#  & K
 \end{pmatrix}$.
\end{center}
Next, let $D^\#$ exist. Then the complementary pseudo principal pivot transform of $M$ relative to $D$ is defined by
\begin{center}
$T=cpppt(M,D)_{\#}=
\begin{pmatrix}
 L & BD^\# \\
-D^\# C & D^\#
\end{pmatrix}$. 
\end{center} 

\section{Vanishing Pseudo Schur Complements}
Let $A\in \mathbb{R}^{m\times n}$, $B\in \mathbb{R}^{m\times p}$, $C\in \mathbb{R}^{s\times n}$, $D\in \mathbb{R}^{s\times p}$ and $M=\begin{pmatrix}
A & B\\
C & D
\end{pmatrix}$. As before, let $F=D-CA^\dagger B$ be the pseudo Schur complement (of $A$ in $M$) and $G=A-BD^\dagger C$ be the complementary pseudo Schur complement (of $D$ in $M$). The implication $G=0\implies F=0$ has applications in studying the cancellation property of a product of three matrices and has been explored in \cite{damm2009cancellation}, where the question of characterizing such an implication was left as an open problem. A characterization was proved recently, using the singular value decomposition \cite{wang2011associated}. In this section, we present a much simpler and conceptual linear algebraic proof which, in particular, does not use the singular value decomposition. One clear advantage of this approach is that such a proof technique extends immediately to operators on Hilbert spaces. However, this will not be our concern in the present work. 

The reverse order law for the product of two matrices $A$ and $B$ is $(AB)^\dagger=B^\dagger A^\dagger$. It is well known that this law is not true, in general. The problem of when it holds was first studied by Greville \cite{greville1966note} who showed that a necessary and sufficient for the law to hold is that $BB^{*} A^\dagger A$ and $A^{*} ABB^\dagger$ are Hermitian. This was simplified by Argirhiade who showed that the reverse law order holds if and only if $R(A^*ABB^*)=R(BB^*A^*A)$ (see the references in \cite{hartwig1986reverse}). Hartwig studied this law for the product of three matrices and derived many characterizations for the formula $(ABC)^\dagger=C^\dagger B^\dagger A^\dagger$ to hold \cite{hartwig1986reverse}. 

We start with a result that is used to prove the main theorem. It shows that when the pseudo Schur complements are both zero, the submatrices constituting the block matrix $M$ satisfy certain reverse order laws for products of three matrices. 

\begin{thm}\label{revlaw1}
Let $M$ be defined as above. Suppose that $F=0$ and $G=0$. Then the following reverse order laws hold: \\
$(i)~ (BD^\dagger C)^\dagger=C^\dagger D B^\dagger$.\\
$(ii)~ (CA^\dagger B)^\dagger=B^\dagger AC^\dagger$.
\end{thm}
\begin{proof}
$(i)$: First, since $G=0$, we have $A=BD^\dagger C$. We then have $R(A)\subseteq R(B)$ and $N(C)\subseteq N(A)$. Thus $BB^\dagger A=A$. The inclusion, $N(C)\subseteq N(A)$ is equivalent to $R(A^*)\subseteq R(C^*)$ and so, $AC^\dagger C=A$. Similarly, since $D=CA^\dagger B$, we have $R(D)\subseteq R(C)$ and $N(B)\subseteq N(D)$. Thus $CC^\dagger D=D$ and $DB^\dagger B=D$. Now, $C^\dagger D=C^\dagger CA^\dagger B$. We have 
\begin{center}
$R(A^{\dagger}) =R(A^*) \subseteq R(C^*)=R(C^{\dagger})=R(C^{\dagger}C)$.
\end{center}
So, $C^\dagger CA^\dagger =A^\dagger$ and thus $C^\dagger D = A^\dagger B$. Let $Y=C^\dagger DB^\dagger$. Then 
\begin{eqnarray*}
AY&=&AC^\dagger DB^\dagger\\
&=&AA^\dagger BB^\dagger\\
&=&(AA^\dagger)^*(BB^\dagger)^*\\
&=&(BB^\dagger AA^\dagger)^*\\
&=&(AA^\dagger)^*\\
&=&AA^\dagger.
\end{eqnarray*}
So, $AYA=A$ and $(AY)^*=AY$. Also, 
\begin{eqnarray*}
YA&=&C^\dagger DB^\dagger A\\
&=&A^\dagger BB^\dagger A\\
&=&A^\dagger A,
\end{eqnarray*}
where we have used the fact that $BB^\dagger A=A$. Thus $(YA)^*=YA$. Further, 
\begin{eqnarray*}
YAY&=&A^\dagger AC^\dagger DB^\dagger\\
&=&A^\dagger AA^\dagger BB^\dagger\\
&=&A^\dagger BB^\dagger\\
&=&C^\dagger DB^\dagger\\
&=&Y.
\end{eqnarray*}
Thus $C^\dagger DB^\dagger=Y=A^\dagger=(BD^\dagger C)^\dagger$.\\
$(ii)$: Similar to the proof of part $(i)$.
\end{proof}

Next, we prove a characterization for the implication $G=0\implies F=0$ to hold. As mentioned earlier, this result was proved in \cite{wang2011associated} (Theorem 1). We present a simpler proof.

\begin{thm}\label{gimpliesf}
Let $M$ be defined as earlier. Suppose that $G=0$. Then $F=0$ if and only if the matrices $B,C$ and $D$ satisfy $(I-CC^\dagger)D=0$, $D(I-B^\dagger B)=0$ and $(BD^\dagger C)^\dagger =C^\dagger DB^\dagger$.
\end{thm} 
\begin{proof}
Necessity: Suppose that $F=0$ so that $D=CA^\dagger B$. Then 
\begin{center}
$CC^\dagger D= CC^\dagger CA^\dagger B=CA^\dagger B =D$
\end{center}
and so $(I-CC^\dagger)D=0$. Again, one has 
\begin{center}
$DB^\dagger B=CA^\dagger B B^\dagger B=CA^\dagger B=D$
\end{center}
and so $D(I-B^\dagger B) = 0$. Finally, since $G=0$, by Theorem \ref{revlaw1}, it follows that $(BD^\dagger C)^\dagger=C^\dagger DB^\dagger$. 

Sufficiency: We have $CC^\dagger D=D=DB^\dagger B$ and $C^\dagger DB^\dagger=(BD^\dagger C)^\dagger$. Since $G=0$ we have $C^\dagger DB^\dagger=A^\dagger$. Hence, 
\begin{eqnarray*}
F&=&D-CA^\dagger B\\
&=&CC^\dagger D-CA^\dagger B\\
&=&CC^\dagger DB^\dagger B-CA^\dagger B\\
&=&CA^\dagger B-CA^\dagger B\\
&=&0.
\end{eqnarray*}
\end{proof}
The reverse implication $F=0\implies G=0$ is considered next. The proof is similar to the proof of Theorem \ref{gimpliesf} and is omitted. 

\begin{thm}\label{fimpliesg}
Let $M$ be defined as earlier. Suppose that $F=0$. Then $G=0$ if and only if $(I-BB^\dagger)A=0$, $A(I-C^\dagger C)=0$ and $(CA^\dagger B)^\dagger =B^\dagger AC^\dagger$.
\end{thm}

Next, we consider an analogue of Theorem \ref{revlaw1} for the group inverse. Let $M=\begin{pmatrix}
A & B\\
C & D
\end{pmatrix}$ such that $A$, $B$, $C$ and $D$ are ${n\times n}$ group invertible matrices. Set $K=D-CA^{\#}B$ and $L=A-BD^{\#}C$.

\begin{thm}\label{revlaw1g}
Let $A$, $B$, $C$, $D$ and $M$ be as above. Suppose that $K=0$ and $L=0$. Then the following reverse order laws hold:\\
$(i)~ C^\# D B^\# \in (BD^\# C)\{1,2\}$.\\
$(ii)~ B^\# AC^\# \in (CA^\# B)\{1,2\}$.
\end{thm}
\begin{proof}
$(i)$: Since $L=0$, we have $A=BD^\#C$. So, $R(A)\subseteq R(B)$ and $N(C)\subseteq N(A)$. Thus $BB^\# A=A$. Since $N(C)\subseteq N(A)$ implies $R(A^*)\subseteq R(C^*)$. So, $C^*(C^*)^\# A^*=A^*$, which using $(C^*)^\#=(C^\#)^*$ and upon taking conjugates gives, $AC^\# C=A$. Similarly, since $D=CA^\# B$, we have $R(D)\subseteq R(C)$ and $N(B)\subseteq N(D)$. Thus $CC^\# D=D$ and $DB^\# B=D$. Let $Y=C^\#DB^\#$ and $X=BD^{\#}C$. Then
\begin{eqnarray*}
XYX&=&BD^{\#}CC^{\#}DB^{\#}BD^{\#}C \\
&=&BD^{\#}DB^{\#}BD^{\#}C \\
&=&BD^{\#}DD^{\#}C\\
&=& BD^{\#}C\\
&=& X
\end{eqnarray*}
and 
\begin{eqnarray*}
YXY&=&C^{\#}DB^{\#}BD^{\#}CC^{\#}DB^{\#}\\
&=&C^{\#}DD^{\#}DB^{\#} \\
&=& C^{\#}DB^{\#}\\
&=& Y
\end{eqnarray*}

$(ii)$: Similar to the proof of part $(i)$.
\end{proof}

\begin{rem}
In Theorem $\ref{revlaw1g}$, we do not know if one could conclude that $C^{\#}DB^{\#}=(BD^{\#}C)^{\#}$ and $B^{\#}AC^{\#}=(CA^{\#}B)^{\#}$.
\end{rem}

In the next result, we give a characterization for the implication $L=0 \implies K=0$ to hold.

\begin{thm}\label{limpliesk}
Let $M$ be defined as earlier such that $A$, $B$, $C$ and $D$ are group invertible matrices of the same order. Suppose that $L=0$. If $K=0$ then the matrices $B,C$ and $D$ satisfy $(I-CC^\#)D=0$, $D(I-B^\# B)=0$ and $C^\# D B^\# \in (BD^\# C)\{1,2\}$.
\end{thm}
\begin{proof}
Suppose that $K=0$ so that $D=CA^\# B$. Then 
\begin{center}
$CC^\# D= CC^\# CA^\# B=CA^\# B =D$
\end{center}
and so $(I-CC^\#)D=0$. Again, one has 
\begin{center}
$DB^\# B=CA^\# B B^\# B=CA^\# B=D$
\end{center}
and so $D(I-B^\# B) = 0$. Since $L=0$, by Theorem \ref{revlaw1g}, it follows that $C^\# D B^\# \in (BD^\# C)\{1,2\}$.
\end{proof}
Next, we give an example to show that the converse of Theorem \ref{limpliesk} is not true.
\begin{ex}
Let  $A=\begin{pmatrix}
-2 & 1\\
-2 & 1
\end{pmatrix}$, $B=\begin{pmatrix}
1 & 1\\
1 & 1
\end{pmatrix}$, $C=\begin{pmatrix}
2 & -1\\
-1 & 2
\end{pmatrix}$, $D=\begin{pmatrix}
-1 & -1\\
0 & 0
\end{pmatrix}$. Then $L=A-BD^{\#}C=0$, where $D^{\#}=-\frac{1}{2}\begin{pmatrix}
1 & 0\\
1 & 0
\end{pmatrix}$. Also, $R(D^*)\subseteq R(B^*)$ and $R(D)\subseteq R(C)$. Now, $A^{\#}=\frac{1}{10}\begin{pmatrix}
-2 & -2\\
1 & 1
\end{pmatrix}$, $B^{\#}=\frac{1}{4}\begin{pmatrix}
1 & 1\\
1 & 1
\end{pmatrix}$ and $C^{-1}=C^{\#}=\frac{1}{3}\begin{pmatrix}
2 & 1\\
1 & 2
\end{pmatrix}$. Then $C^{\#}DB^{\#}=
-\frac{1}{6}\begin{pmatrix}
2 & 2\\
1 & 1
\end{pmatrix}$. Also, $C^\# D B^\# \in (BD^\# C)\{1,2\}$ but $K=D-CA^{\#}B=-\frac{4}{5}\begin{pmatrix}
0 & 0\\
1 & 1
\end{pmatrix}\neq 0$
\end{ex}
The reverse implication $K=0 \implies L=0$ is considered next. The proof is similar to Theorem \ref{limpliesk}.

\begin{thm}\label{kimpliesl}
Let $M$ be defined as earlier. Suppose that $K=0$. If $L=0$ then $(I-BB^\#)A=0$, $A(I-C^\# C)=0$ and $B^\# AC^\# \in (CA^\# B)\{1,2\}$.
\end{thm}

\section{An Inheritance Property of Pseudo Schur Complements}
The intention here is to prove an inheritance property for the pseudo Schur complement and the complementary Schur complement. The property under consideration arises from the absorption law, rather well studied in the literature. For invertible elements $a,b$ in a ring $\cal R$ with unity, one has the equality $a^{-1} + b^{-1}=a^{-1}(a+b) b^{-1}$. It may be easily shown that this does not hold if the inverses are replaced by any one of the many generalized inverses. The question of characterizing when such formulae hold has been well investigated in the past few years. We shall be interested in the matrix version of the absorption law, where the usual inverse is replaced by either the Moore-Penrose inverse or the group inverse. In the first subsection, we study the absorption law for the Moore-Penrose inverse and in the second, the case of the group inverse is considered. It is noteworthy that all the proofs here are purely linear algebraic and so extensions of these results to operators between Hilbert 
spaces are almost immediate. 

\subsection{The Case of the Moore-Penrose Inverse}
First we prove a version of the absorption law expressed in terms of the Moore-Penrose inverse.

\begin{thm}\label{mpal}
Let $A,B \in \mathbb{C}^{m \times n}$. Then the following statments are equivalent:\\
$(i)$~$A^\dagger+B^\dagger=A^\dagger(A+B)B^\dagger$\\
$(ii)$~$R(B^*)\subseteq R(A^*)$ and $R(A)\subseteq R(B)$.
\end{thm}
\begin{proof}
$(i) \Longrightarrow (ii)$: Let $y \in R(B^*)=R(B^\dagger)$. Then $y=B^\dagger x=(A^\dagger (A+B)B^\dagger-A^\dagger)x,~x\in \mathbb{C}^{m}$. Thus $y \in R(A^\dagger)=R(A^*)$. So, $R(B^*)\subseteq R(A^*)$. Again, let $y\in R(A)=R((A^*)^\dagger)$. Then
\begin{eqnarray*}
y=(A^\dagger)^*x &=&(A^\dagger(A+B)B^\dagger-B^\dagger)^*x\\
                 &=& (B^\dagger)^*(A^*+B^*)(A^\dagger)^*x-(B^\dagger)^*x.
\end{eqnarray*}
So, $y \in R((B^\dagger)^*)=R(B)$. Hence, $R(A)\subseteq R(B)$.\\

$(ii)\Longrightarrow (i)$: We show that $A^\dagger AB^\dagger=B^\dagger$ and $A^\dagger BB^\dagger=A^\dagger$. Since $R(B^*)\subseteq R(A^*)$, we have $BA^\dagger A=B$. Set $X=A^\dagger AB^\dagger$. We show that $X=B^\dagger$. First, $BX=BA^\dagger AB^\dagger=BB^\dagger$. So $BXB=B$ and $(BX)^*=BX$. Further, 
\begin{eqnarray*}
XB&=&A^\dagger AB^\dagger B \\
 &=& (A^\dagger A)^*(B^\dagger B)^*\\
 &=&(B^\dagger BA^\dagger A)^*\\
 &=&(B^\dagger B)^*\\
 &=&B^\dagger B.
\end{eqnarray*}
Thus $(XB)^*=XB$. Also, $XBX=XBB^\dagger=A^\dagger AB^\dagger BB^\dagger=A^\dagger AB^\dagger=X$. Thus, $X=B^\dagger$.\\
Set $Y=A^\dagger BB^\dagger$. We show that $Y=A^\dagger$. Since $R(A)\subseteq R(B)$, we have $BB^\dagger A=A$. Then 
\begin{eqnarray*}
AY&=&AA^\dagger BB^\dagger\\
 &=&(AA^\dagger)^*(BB^\dagger)^*\\
 &=&(BB^\dagger AA^\dagger)^*\\
 &=&(AA^\dagger)^*\\
 &=&AA^\dagger.
\end{eqnarray*} Thus, $AYA=A$ and $(AY)^*=AY$. Also, $YA=A^\dagger BB^\dagger A=A^\dagger A$. Thus, $(YA)^*=YA$ and 
\begin{eqnarray*}
YAY &=&A^\dagger AA^\dagger BB^\dagger\\
&=&A^\dagger BB^\dagger\\
&=&Y.
\end{eqnarray*}
Thus $Y=A^\dagger$. Now, we have 
\begin{eqnarray*}
A^\dagger(A+B)B^\dagger&=&A^\dagger AB^\dagger+A^\dagger BB^\dagger\\
&=&X + Y\\
&=&B^\dagger+A^\dagger,
\end{eqnarray*}
completing the proof.
\end{proof}

Next, we give an example of Moore-Penrose absorbing pair.

\begin{ex}
Let $A=\begin{pmatrix}
\frac{1}{2} & \frac{1}{2} & 0\\
-2 & 1 & 0\\
0 & 0 & 0
\end{pmatrix}$ and $B=\begin{pmatrix}
2 & 1 & 0\\
2 & 2 & 0\\
0 & 0 & 0
\end{pmatrix}$. Then $A^\dagger=\begin{pmatrix}
\frac{2}{3} & \frac{1}{3} & 0\\
\frac{4}{3} & \frac{1}{3} & 0\\
0 & 0 & 0
\end{pmatrix}$ and $B^\dagger=\begin{pmatrix}
1 & -\frac{1}{2} & 0\\
-1 & 1 & 0\\
0 & 0 & 0
\end{pmatrix}$. Here, $R(B^*)\subseteq R(A^*)$ and $R(A)\subseteq R(B)$. Also, we can verify that $A^\dagger+B^\dagger=A^\dagger(A+B)B^\dagger$.
\end{ex}

\begin{defn}
Let $A,B \in \mathbb{C}^{m \times n}$. The pair $(A,B)$ is called a {\it Moore-Penrose inverse absorbing pair} if $A^\dagger+B^\dagger=A^\dagger(A+B)B^\dagger$.
\end{defn}

By Theorem \ref{mpal}, it follows that $(A,B)$ is a Moore-Penrose inverse absorbing pair if and only if $R(A) \subseteq R(B)$ and $R(B^*) \subseteq R(A^*)$. Clearly, $(A,A)$ is a Moore-Penrose inverse absorbing pair and that if either $(A,0)$ or $(0,A)$ is a Moore-Penrose inverse absorbing pair, then $A=0$. Also, $(A,B)$ is a Moore-Penrose inverse absorbing pair if and only if so is $(B^*,A^*)$. This, in turn holds if and only if $(B^{\dagger},A^{\dagger})$ is a Moore-Penrose inverse absorbing pair. It is also clear that if $(A,B)$ and $(B,C)$ are Moore-Penrose inverse absorbing pairs, then $(A,C)$ is a Moore-Penrose inverse absorbing pair. In particular, if we define a relation ``m'' on $\mathbb{C}^{m \times n}$ by $AmB$ if and only if $(A,B)$ is a Moore-Penrose inverse absorbing pair, then ``m'' is a reflexive and transitive relation. It is not difficult to observe that such a relation will be a partial order relation on the set of orthogonal projections on $\mathbb{C}^{n \times n}$. 

In what follows, we address the question of when the pair of pseudo Schur complements of a Moore-Penrose inverse absorbing pair of two block matrices inherits that property. 

\begin{thm}\label{abs_schur}
Let $U=\begin{pmatrix}
A_{U} & B_{U}\\
C_{U} & D_{U}
\end{pmatrix}$ and $V=\begin{pmatrix}
A_{V} & B_{V}\\
C_{V} & D_{V}
\end{pmatrix}$. Suppose that $R(B_{U})\subseteq R(A_{U})$, $R((A_{U}^\dagger B_{U})^*)\subseteq R(F_{U}^*)$, $R(C_{V}^*)\subseteq R(A_{V}^*)$ and $R(C_{V}A_{V}^\dagger)\subseteq R(F_{V})$, where $F_{U}=D_{U}-C_{U}A_{U}^\dagger B_{U}$ and $F_{V}=D_{V}-C_{V}A_{V}^\dagger B_{V}$. If $(U,V)$ is a Moore-Penrose inverse absorbing pair then $(F_{U},F_{V})$ is also a Moore-Penrose inverse absorbing pair.
\end{thm}
\begin{proof}
Since $(U,V)$ is a Moore-Penrose inverse absorbing pair, by Theorem \ref{mpal}, we then have $R(U)\subseteq R(V)$ and $R(V^{*})\subseteq R(U^{*})$. Let $x\in N(F_{U})$. Then 
\begin{center}
$0=F_{U}x=(D_{U}-C_{U}A_{U}^\dagger B_{U})x$.                                                                                                                                                                 \end{center}
Also, $N(F_{U})\subseteq N(A_{U}^\dagger B_{U})$. So, $D_{U}x=0$. Finally, since $A_{U}^\dagger B_{U}x=0$, on premultiplying by $A_{U}$ and using the fact that $A_{U}A_{U}^\dagger B_{U}=B_{U}$, we get $B_{U}x=0$. Set $z=\begin{pmatrix}
0\\
x
\end{pmatrix}$. Then $Uz=\begin{pmatrix}
B_{U}x\\
D_{U}x
\end{pmatrix}=0$. Thus $z\in N(U)$ and hence $x\in N(V)$. So 
\begin{center}
$0=Vz=\begin{pmatrix}
A_{V} & B_{V}\\
C_{V} & D_{V}
\end{pmatrix}
\begin{pmatrix}
0\\
x
\end{pmatrix}
=
\begin{pmatrix}
B_{V}x\\
D_{V}x
\end{pmatrix}$,
\end{center}
so that $B_{V}x=0$ and $D_{V}x=0$. Thus, 
\begin{center}
$F_{V}x=(D_{V}-C_{V}A_{V}^\dagger B_{V})x=0$
\end{center}
and hence $N(F_{U})\subseteq N(F_{V})$.\\
Next, we prove that $N(F_{V}^*)\subseteq N(F_{U}^*)$. Let $x\in N(F_{V}^*)$. Then 
\begin{center}
$0=F_{V}^*x=(D_{V}-C_{V}A_{V}^\dagger B_{V})^*x$.
\end{center}
Using the fact that $N(F_{V}^*)\subseteq N((C_{V}A_{V}^\dagger)^*)$, we then have $(C_{V}A_{V}^\dagger)^*x=0$ and hence $D_{V}^*x=B_{V}^*(C_{V}A_{V}^\dagger)^*x=0$. Since $R(C_{V}^*)\subseteq R(A_{V}^*)$, we have $C_{V}A_{V}^\dagger A_{V}=C_{V}$. So, upon premultiplying $(C_{V}A_{V}^\dagger)^*x=0$ by $A_{V}^*$, we have 
\begin{center}
$0=A_{V}^*(C_{V}A_{V}^\dagger)^*x=(C_{V}A_{V}^\dagger A_{V})^*x=C_{V}^*x$.
\end{center}
Set $y=\begin{pmatrix}
0\\
x
\end{pmatrix}$. Then 
\begin{center}
$V^Ty=\begin{pmatrix}
A_{V}^* & C_{V}^*\\
B_{V}^* & D_{V}^*
\end{pmatrix}
\begin{pmatrix}
0\\
x
\end{pmatrix}=
\begin{pmatrix}
C_{V}^*x\\
D_{V}^*x
\end{pmatrix}=\begin{pmatrix}
0\\
0
\end{pmatrix}$. 
\end{center}
Thus $y\in N(V^*)$. Since $(U,V)$ is a Moore-Penrose inverse absorbing pair, we have $N(V^*) \subseteq N(U^*)$ and thus $y\in N(U^*)$. So 
\begin{center}
$0=U^*y=\begin{pmatrix}
A_{U}^* & C_{U}^*\\
B_{U}^* & D_{U}^*
\end{pmatrix}
\begin{pmatrix}
0\\
x
\end{pmatrix}=\begin{pmatrix}
C_{U}^*x\\
D_{U}^*x
\end{pmatrix}$
\end{center}
and so $C_{U}^*x=0$ and $D_{U}^*x=0$. Hence 
\begin{center}
$F_{U}^*x=D_{U}^*x-(A_{U}^\dagger B_{U})^*C_{U}^*x=0$.
\end{center}
Thus $x\in N(F_{U}^*)$ proving that $N(F_{V}^*)\subseteq N(F_{U}^*)$. By Thoerem \ref{mpal}, it follows that $(F_{U},F_{V})$ is a Moore-Penrose inverse absorbing pair.
\end{proof}

An analogoue for the pair of complementary pseudo Schur complements is given next. The proof is similar to the proof of Theorem \ref{abs_schur} and is omitted.

\begin{thm}
Let $U$ and $V$ be defined as above. Suppose that $R(C_{U})\subseteq R(D_{U})$, $R((D_{U}^\dagger C_{U})^*)\subseteq R(G_{U}^*)$, $R(B_{V}^*)\subseteq R(D_{V}^*)$ and $R(B_{V}D_{V}^\dagger)\subseteq R(G_{V})$, where $G_{U}=A_{U}-B_{U}D_{U}^\dagger C_{U}$ and $G_{V}=A_{V}-B_{V}D_{V}^\dagger C_{V}$. If $(U,V)$ is a Moore-Penrose inverse absorbing pair then $(G_{U},G_{V})$ is also a Moore-Penrose inverse absorbing pair.
\end{thm}

\subsection{The Group Inverse Version}
In this second part, we present versions of the results of the previous subsection for the group inverse. First we prove a theorem on the absorption law for the group inverse, analogous to Theorem \ref{mpal}.

\begin{thm}\label{gpal}
Let $A, B \in \mathbb{C}^{n \times n}$ have index $1$. Then the following statements are equivalent:\\
$(i)~A^\#+B^\#=A^\#(A+B)B^\#$.\\
$(ii)~R(A^*)\subseteq R(B^*)$ and $R(B)\subseteq R(A)$.
\end{thm}
\begin{proof}
$(i)\implies(ii)$: Let $y \in R(B)=R(B^\#)$. Then for some $x\in \mathbb{C}^{n}$ one has 
\begin{center}
$y =  B^\# x = (A^\# (A+B)B^\#-A^\#)x$.
\end{center}
Thus $y \in R(A^\#)=R(A)$. So, $R(B)\subseteq R(A)$. Again, let $y\in R(A^{*})=R((A^\#)^*)$. Then
\begin{eqnarray*}
y&=&(A^\#)^*x \\
&=&(A^\#(A+B)B^\#-B^\#)^*x\\
&=&((B^\#)^*-(B^\#)^*(A^*+B^*)(A^\#)^*)x.
\end{eqnarray*}
So, $y \in R((B^\#)^*)=R(B^*)$. Hence, $R(A^*)\subseteq R(B^*)$.\\
$(ii)\implies(i)$: Since $R(B^\#)=R(B)\subseteq R(A)=R(A^\#)$, we then have $A^\# AB^\#=B^\#$. Also, 
\begin{center}
$R((A^*)^\#)=R((A^\#)^*)=R(A^*)\subseteq R(B^*)=R((B^\#)^*)=R((B^*)^\#)$
\end{center}
So, by what we discussed just now, we have  $(B^*)^\# B^*(A^*)^\#=(A^*)^\#.$ Thus, $(B^\#)^*B^*(A^\#)^*=(A^\#)^*.$ Upon taking conjugates, we get 
$A^\# BB^\#=A^\#$. So, 
\begin{eqnarray*}
A^\#(A+B)B^\#&=&A^\# AB^\#+A^\# BB^\# \\
&=&B^\#+A^\#,
\end{eqnarray*}
completing the proof.
\end{proof}

\begin{defn}
Let $A,B \in \mathbb{C}^{n \times n}$ both be group invertible. The pair $(A,B)$ is called a {\it group inverse absorbing pair} if $A^\#+B^\#=A^\#(A+B)B^\#$.
\end{defn}

By Theorem \ref{gpal}, it follows that $(A,B)$ is a group inverse absorbing pair if and only if $R(A^*) \subseteq R(B^*)$ and $R(B) \subseteq R(A)$. Clearly, if $A$ is group invertible, then $(A,A)$ is a group inverse absorbing pair. If $(A,0)$ or $(0,A)$ are group inverse absorbing pairs, then $A=0$. Also, $(A,B)$ is a group inverse absorbing pair if and only if so is $(B^*,A^*)$. This, in turn holds if and only if $(A^{\#},B^{\#})$ is a group inverse absorbing pair. This marks a departure from the corresponding result for the Moore-Penrose inverse. It is also clear that if $(A,B)$ and $(B,C)$ are group inverse absorbing pairs, then $(A,C)$ is a group inverse absorbing pair. In particular, if we define a relation ``g'' on the subset of all group invertible matrices in $\mathbb{C}^{n \times n}$ by $AgB$ if and only if $(A,B)$ is a group inverse absorbing pair, then ``g'' is a reflexive and transitive relation. It is also easy to see that ``g'' will be a partial order relation on the subset of all idempotent 
matrices in $\mathbb{C}^{n \times n}$.

Next, we obtain a group inverse analogue of Theorem \ref{abs_schur}. 

\begin{thm}
Let $U=\begin{pmatrix}
A_{U} & B_{U}\\
C_{U} & D_{U}
\end{pmatrix}$ and $V=\begin{pmatrix}
A_{V} & B_{V}\\
C_{V} & D_{V}
\end{pmatrix}$. Assume that $A_{U}^{\#}$ and $A_{V}^{\#}$ exist. Suppose that $R(B_{V})\subseteq R(A_{V})$, $R((A_{V}^\# B_{V})^*)\subseteq R(K_{V}^*)$, $R(C_{U}^*)\subseteq R(A_{U}^*)$ and $R(C_{U}A_{U}^\#)\subseteq R(K_{U})$, where $K_{U}=D_{U}-C_{U}A_{U}^\# B_{U}$ and $K_{V}=D_{V}-C_{V}A_{V}^\# B_{V}$. If $(U,V)$ is a group inverse absorbing pair then $(K_{U},K_{V})$ is also a group inverse absorbing pair.
\end{thm}

\begin{proof}
Since $(U,V)$ is a group inverse absorbing pair, by Theorem \ref{gpal}, we then have $R(V)\subseteq R(U)$ and $R(U^{*})\subseteq R(V^{*})$. Let $x\in N(K_{V})$. Then 
\begin{center}
$0=K_{V}x=(D_{V}-C_{V}A_{V}^\# B_{V})x$.                                                                                                                                                                
\end{center}
Also, $N(K_{V})\subseteq N(A_{V}^\# B_{V})$. So, $D_{V}x=0$. Finally, since $A_{V}^\# B_{V}x=0$, on premultiplying by $A_{V}$ and using the fact that $A_{V}A_{V}^\# B_{V}=B_{V}$, we get $B_{V}x=0$. Set $z=\begin{pmatrix}
0\\
x
\end{pmatrix}$. Then $Vz=\begin{pmatrix}
B_{V}x\\
D_{V}x
\end{pmatrix}=0$. Thus $z\in N(V)$ and hence $x\in N(U)$. So 
\begin{center}
$0=Uz=\begin{pmatrix}
A_{U} & B_{U}\\
C_{U} & D_{U}
\end{pmatrix}
\begin{pmatrix}
0\\
x
\end{pmatrix}
=
\begin{pmatrix}
B_{U}x\\
D_{U}x
\end{pmatrix}$,
\end{center}
so that $B_{U}x=0$ and $D_{U}x=0$. Thus, 
\begin{center}
$K_{U}x=(D_{U}-C_{U}A_{U}^\# B_{U})x=0$
\end{center}
and hence $N(K_{V})\subseteq N(K_{U})$.\\
Next, we prove that $N(K_{U}^*)\subseteq N(K_{V}^*)$. Let $x\in N(K_{U}^*)$. Then 
\begin{center}
$0=K_{U}^*x=(D_{U}-C_{U}A_{U}^\# B_{U})^*x$.
\end{center}
Using the fact that $N(K_{U}^*)\subseteq N((C_{U}A_{U}^\#)^*)$, we then have $(C_{U}A_{U}^\#)^*x=0$ and hence $D_{U}^*x=B_{U}^*(C_{U}A_{U}^\#)^*x=0$. Since $R(C_{U}^*)\subseteq R(A_{U}^*)$, we have $C_{U}A_{U}^\# A_{U}=C_{U}$. So, upon premultiplying $(C_{U}A_{U}^\#)^*x=0$ by $A_{U}^*$, we have 
\begin{center}
$0=A_{U}^*(C_{U}A_{U}^\#)^*x=(C_{U}A_{U}^\# A_{U})^*x=C_{U}^*x$.
\end{center}
Set $y=\begin{pmatrix}
0\\
x
\end{pmatrix}$. Then 
\begin{center}
$U^*y=\begin{pmatrix}
A_{U}^* & C_{U}^*\\
B_{U}^* & D_{U}^*
\end{pmatrix}
\begin{pmatrix}
0\\
x
\end{pmatrix}=
\begin{pmatrix}
C_{U}^*x\\
D_{U}^*x
\end{pmatrix}=\begin{pmatrix}
0\\
0
\end{pmatrix}$. 
\end{center}
Thus $y\in N(U^*)$. Since $(U,V)$ is a group inverse absorbing pair, we have $N(U^*) \subseteq N(V^*)$ and thus $y\in N(V^*)$. So 
\begin{center}
$0=V^*y=\begin{pmatrix}
A_{V}^* & C_{V}^*\\
B_{V}^* & D_{V}^*
\end{pmatrix}
\begin{pmatrix}
0\\
x
\end{pmatrix}=\begin{pmatrix}
C_{V}^*x\\
D_{V}^*x
\end{pmatrix}$
\end{center}
and so $C_{V}^*x=0$ and $D_{V}^*x=0$. Hence 
\begin{center}
$K_{V}^*x=D_{V}^*x-(A_{V}^\# B_{V})^*C_{V}^*x=0$.
\end{center}
Thus $x\in N(K_{V}^*)$ proving that $N(K_{U}^*)\subseteq N(K_{V}^*)$. By Thoerem \ref{gpal}, it follows that $(K_{U},K_{V})$ is a group inverse absorbing pair.
\end{proof}

An analogous result for the complementary pseudo Schur complements is given next. The proof is omitted.

\begin{thm}
Let $U$ and $V$ be defined as above. Assume that $D_{U}^{\#}$ and $D_{V}^{\#}$ exist. Suppose that $R(C_{V})\subseteq R(D_{V})$, $R((D_{V}^\# C_{V})^*)\subseteq R(L_{V}^*)$, $R(B_{U}^*)\subseteq R(D_{U}^*)$ and $R(B_{U}D_{U}^\#)\subseteq R(L_{U})$, where $L_{V}=A_{V}-B_{V}D_{V}^\# C_{V}$ and $L_{U}=A_{U}-B_{U}D_{U}^\# C_{U}$. If $(U,V)$ is a group inverse absorbing pair then $(L_{U},L_{V})$ is also a group inverse absorbing pair.
\end{thm}

\section{Inheritance by Pseudo Principal Pivot Transform}
In this last section, we discuss inheritance of the absorption law by the pseudo principal pivot transforms. First we prove results for principal subblocks.

\subsection{The Case of the Moore-Penrose Inverse}

\begin{thm}\label{prinsub1}
Let $U$ and $V$ be as above. Suppose that $R(B_{V})\subseteq R(A_{V})$ and $R(C_{U}^*)\subseteq R(A_{U}^*)$. If $(U,V)$ is a Moore-Penrose inverse absorbing pair then $(A_{U},A_{V})$ is also a Moore-Penrose inverse absorbing pair.
\end{thm}
\begin{proof}
Let $x\in N(A_{U})$. Then $x\in N(C_{U})$. Set $z=\begin{pmatrix}
x\\
0
\end{pmatrix}$. Then 
\begin{center}
$Uz=\begin{pmatrix}
A_{U} & B_{U}\\
C_{U} & D_{U}
\end{pmatrix}
\begin{pmatrix}
x\\
0
\end{pmatrix}=
\begin{pmatrix}
A_{U}x\\
C_{U}x
\end{pmatrix}=0$.
\end{center}
Thus $z\in N(U)$. Since $(U,V)$ is a Moore-Penrose inverse absorbing pair, we then have $z\in N(V)$. So, \begin{center}
$0=Vz=\begin{pmatrix}
A_{V} & B_{V}\\
C_{V} & D_{V}
\end{pmatrix}\begin{pmatrix}
x\\
0
\end{pmatrix}$. 
\end{center}
Then $A_{V}x=0$ and hence $N(A_{U})\subseteq N(A_{V})$, i.e., $R(A_{V}^*)\subseteq R(A_{U}^*)$. Next, we prove that $R(A_{U})\subseteq R(A_{V})$. We prove $N(A_{V}^*)\subseteq N(A_{U}^*)$, instead. Let $x\in N(A_{V}^*)$. Then $x\in N(B_{V}^*)$ as $N(A_{V}^*)\subseteq N(B_{V}^*)$. Set $z=\begin{pmatrix}
x\\
0
\end{pmatrix}$. We then have 
\begin{center}
$V^*z=\begin{pmatrix}
A_{V}^* & C_{V}^*\\
B_{V}^* & D_{V}^*
\end{pmatrix}
\begin{pmatrix}
x\\
0
\end{pmatrix}=
\begin{pmatrix}
A_{V}^*x\\
B_{V}^*x
\end{pmatrix}$.
\end{center}
Thus $z\in N(V^*)$ so that $z\in N(U^*)$. So, 
\begin{center}
$0=U^*z=\begin{pmatrix}
A_{U}^* & C_{U}^*\\
B_{U}^* & D_{U}^*
\end{pmatrix}
\begin{pmatrix}
x\\
0
\end{pmatrix}=
\begin{pmatrix}
A_{U}^*x\\
B_{U}^*x
\end{pmatrix}$.
\end{center}
So $x\in N(A_{U}^*)$ and hence $N(A_{V}^*)\subseteq N(A_{U}^*)$, proving that $(A_{U},A_{V})$ is a Moore-Penrose inverse absorbing pair.
\end{proof}

The proof of the next result is similar and is hence omitted. This concerns the inheritance of the second principal diagonal subblock.

\begin{thm}\label{prinsub2}
Let $U$ and $V$ be as above. Suppose that $R(C_{V})\subseteq R(D_{V})$ and $R(B_{U}^*)\subseteq R(D_{U}^*)$. If $(U,V)$ is a Moore-Penrose inverse absorbing pair then $(D_{U},D_{V})$ is also a Moore-Penrose inverse absorbing pair.
\end{thm}

Now, we turn our attention to inheritance by the pseudo principal pivot transform.

\begin{thm}\label{inherpppt}
Let $U$ and $V$ be as above. Suppose that $ R(C_{U}^*)\subseteq R(A_{U}^*)$, $R(B_{U})\subseteq R(A_{U})$, $R(B_{U}^*)\subseteq R(D_{U}^*)$, $R(C_{U})\subseteq R(D_{U})$, $R(C_{V}^*)\subseteq R(A_{V}^*)$, $R(B_{V})\subseteq R(A_{V})$, $R(B_{V}^*)\subseteq R(D_{V}^*)$ and $R(C_{V})\subseteq R(D_{V})$. \\
$(i)$ Suppose that $A_{U}$ and $A_{V}$ are range symmetric. If $(U,V)$ is a Moore-Penrose inverse absorbing pair then $(H_{U},H_{V})$ is also a Moore-Penrose inverse absorbing pair, where 
\begin{center}
$H_{U}=pppt(U,A_{U})_{\dagger}=\begin{pmatrix}
A_{U}^\dagger & -A_{U}^\dagger B_{U}\\
C_{U}A_{U}^\dagger & F_{U}
\end{pmatrix}$ 
\end{center}
and 
\begin{center}
$H_{V}=pppt(V,A_{V})_{\dagger}=\begin{pmatrix}
A_{V}^\dagger & -A_{V}^\dagger B_{V}\\
C_{V}A_{V}^\dagger & F_{V}
\end{pmatrix}$.
\end{center}
$(ii)$ Suppose that $D_{U}$ and $D_{V}$ are range symmetric. If $(U,V)$ is a Moore-Penrose inverse absorbing pair then $(J_{U},J_{V})$ is also a Moore-Penrose inverse absorbing pair, where 
\begin{center}
$J_{U}=pppt(U,D_{U})_{\dagger}=\begin{pmatrix}
G_{U} & B_{U}D_{U}^\dagger \\
-D_{U}^\dagger C_{U} & D_{U}^\dagger
\end{pmatrix}$
\end{center}
and 
\begin{center}
$J_{V}=pppt(V,D_{V})_{\dagger}=\begin{pmatrix}
G_{V} & B_{V}D_{V}^\dagger \\
-D_{V}^\dagger C_{V} & D_{V}^\dagger
\end{pmatrix}$.
\end{center}
\end{thm}
\begin{proof}
$(i)$ Since $(U,V)$ is a Moore-Penrose inverse absorbing pair, from Theorem \ref{mpal}, we have $R(U)\subseteq R(V)$ and $R(V^*)\subseteq R(U^*)$. First, we show that $R(H_{U})\subseteq R(H_{V})$. Let $x=\begin{pmatrix}
x^{1}\\
x^{2}
\end{pmatrix}$ such that $x\in R(H_{U})$. As $R(H_{U})=R(H_{U}H_{U}^\dagger)$, we then have $x \in R(H_{U}H_{U}^\dagger)$. From (Theorem 3.1) \cite{kavitaPseudoschur2015}, 
\begin{center}
$H_{U}^\dagger=\begin{pmatrix}
G_{U} & B_{U}D_{U}^\dagger\\
-D_{U}^\dagger C_{U} & D_{U}^\dagger
\end{pmatrix}$
\end{center}
and so 
\begin{center}
$H_{U}H_{U}^\dagger=\begin{pmatrix}
A_{U}^\dagger A_{U} & 0\\
0 & D_{U}D_{U}^\dagger
\end{pmatrix}$. 
\end{center}
Thus $x^{1}\in R(A_{U}^\dagger A_{U})=R(A_{U}^\dagger)$. Since $A_{U}$ is range symmetric, we then have $x^{1}\in R(A_{U}^\dagger)=R(A_{U}^*)=R(A_{U})$ and $x^{2}\in R(D_{U}D_{U}^\dagger)=R(D_{U})$. From Theorem \ref{prinsub1} and Theorem \ref{prinsub2}, both the pairs $(A_{U},A_{V})$ and $(D_{U},D_{V})$ are Moore-Penrose inverse absorbing and hence $R(A_{U})\subseteq R(A_{V})$ and $R(D_{U})\subseteq R(D_{V})$. Finally, we have $x^{1}\in R(A_{V})=R(A_{V}^*)=R(A_{V}^\dagger A_{V})$ and $x^{2}\in R(D_{V})=R(D_{V}D_{V}^\dagger)$. Since 
\begin{center}
$H_{V}^\dagger=\begin{pmatrix}
G_{V} & B_{V}D_{V}^\dagger\\
-D_{V}^\dagger C_{V} & D_{V}^\dagger
\end{pmatrix}$,
\end{center}
one has 
\begin{center}
$H_{V}H_{V}^\dagger x=\begin{pmatrix}
A_{V}^\dagger A_{V}x^{1} & 0\\
0 & D_{V}D_{V}^\dagger x^{2}
\end{pmatrix}=\begin{pmatrix}
x^{1}\\
x^{2}
\end{pmatrix}=x$.
\end{center}
So $x\in R(H_{V})$. Thus $R(H_{U})\subseteq R(H_{V})$. Next, we prove that $R(H_{V}^*)\subseteq R(H_{U}^*)$. Let $y=\begin{pmatrix} y^{1}\\ y^{2}
\end{pmatrix}$ with $y\in R(H_{V}^*)$. Using the fact that $R(H_{V}^*)=R(H_{V}^\dagger)=R(H_{V}^\dagger H_{V})$, we then have $y\in R(H_{V}^\dagger H_{V})$. By using the expression for $H_{V}^\dagger H_{V}$, it follows that $y^{1}\in R(A_{V}A_{V}^\dagger)=R(A_{V})$. Since $A_{V}$ is range symmetric, we then have $y^{1}\in R(A_{V}^*)$ and $y^{2}\in R(D_{V}^\dagger D_{V})=R(D_{V}^*)$. Since $(A_{U},A_{V})$ is a Moore-Penrose absorbing pair, we have $R(A_{V}^*)\subseteq R(A_{U}^*)$ and hence 
\begin{center}
$y^{1}\in R(A_{U}^*)=R(A_{U})=R(A_{U}A_{U}^\dagger)$.
\end{center}

Again, $(D_{U},D_{V})$ is a Moore-Penrose absorbing pair, we have $R(D_{V}^*)\subseteq R(D_{U}^*)$, so $y^{2}\in R(D_{U}^*)=R(D_{U}^\dagger)=R(D_{U}^\dagger D_{U})$. Since
\begin{center}
$ H_{U}^\dagger H_{U}=\begin{pmatrix}
A_{U} A_{U}^\dagger & 0\\
0 & D_{U}^\dagger D_{U}
\end{pmatrix}$,
\end{center}
one has $y\in R(H_{U}^\dagger H_{U})=R(H_{U}^*)$, proving that  $R(H_{V}^*)\subseteq R(H_{U}^*)$. So $(H_{U},H_{V})$ is a Moore-Penrose inverse absorbing pair.\\
$(ii)$ The proof of this part is similar to part $(i)$.
\end{proof}

Next, we give an example to illustrate Theorem $\ref{inherpppt}$.

\begin{ex}
Let $A_{U}=\begin{pmatrix}
\frac{1}{2} & \frac{1}{2}\\
-2 & 1
\end{pmatrix}$, $B_{U}=\begin{pmatrix}
0\\
0
\end{pmatrix}$, $C_{U}=\begin{pmatrix}
0 & 0
\end{pmatrix}$ and $D_{U}=(0)$. Then 
\begin{center}
$U=\begin{pmatrix}
A_{U} & B_{U}\\
C_{U} & D_{U}
\end{pmatrix}
=\begin{pmatrix}
\frac{1}{2} & \frac{1}{2} & 0\\
-2 & 1 & 0\\
0 & 0 &0
\end{pmatrix}$.
\end{center}
Consider, $A_{V}=\begin{pmatrix}
2 & 1\\
2 & 2
\end{pmatrix}$, $B_{V}=\begin{pmatrix}
0\\
0
\end{pmatrix}$, $C_{V}=\begin{pmatrix}
0 & 0
\end{pmatrix}$ and $D_{V}=(0)$. Then \begin{center}
$V=\begin{pmatrix}
A_{V} & B_{V}\\
C_{V} & D_{V}
\end{pmatrix}=\begin{pmatrix}
2 & 1 & 0\\
2 & 2 & 0\\
0 & 0 & 0
\end{pmatrix}$.
\end{center}
 Clearly, all the conditions in Theorem \ref{inherpppt} are satisfied. Also, $R(U)\subseteq R(V)$ and $R(V^*)\subseteq R(U^*)$. Thus $(U,V)$ is a Moore-Penrose absorbing pair. Now,
\begin{center}
$H_{U}=\begin{pmatrix}
\frac{2}{3} & \frac{1}{3} & 0\\
\frac{4}{3} & \frac{1}{3} & 0\\
0 & 0 & 0
\end{pmatrix}$
\end{center}
 and \begin{center}
$H_{V}=\begin{pmatrix}
1 & -\frac{1}{2} & 0\\
-1 & 1 & 0\\
0 &0 &0
\end{pmatrix}$.
\end{center} 
Then $(H_{U},H_{V})$ is a Moore-Penrose absorbing pair as $R(H_{U})\subseteq R(H_{V})$ and $R(H_{V}^*)\subseteq R(H_{U}^*)$.
\end{ex}

\subsection{The Group Inverse Version}
In this subsection, we discuss the inheritance of the absorption law by the pseudo principal pivot transforms for group invertible matrices.
\begin{thm}\label{prinsub1g}
Let $U$ and $V$ be as above. Suppose that $R(B_{U})\subseteq R(A_{U})$ and $R(C_{V}^*)\subseteq R(A_{V}^*)$. If $(U,V)$ is a group inverse absorbing pair then $(A_{U},A_{V})$ is also a group inverse absorbing pair.
\end{thm}
\begin{proof}
Let $x\in N(A_{V})$. Then $x\in N(C_{V})$. Set $z=\begin{pmatrix}
x\\
0
\end{pmatrix}$. Then 
\begin{center}
$Vz=\begin{pmatrix}
A_{V} & B_{V}\\
C_{V} & D_{V}
\end{pmatrix}
\begin{pmatrix}
x\\
0
\end{pmatrix}=
\begin{pmatrix}
A_{V}x\\
C_{V}x
\end{pmatrix}=0$.
\end{center}
Thus $z\in N(V)$. Since $(U,V)$ is a group inverse absorbing pair, we then have $z\in N(U)$. So, \begin{center}
$0=Uz=\begin{pmatrix}
A_{U} & B_{U}\\
C_{U} & D_{U}
\end{pmatrix}\begin{pmatrix}
x\\
0
\end{pmatrix}$. 
\end{center}
Then $A_{U}x=0$ and hence $N(A_{V})\subseteq N(A_{U})$, i.e., $R(A_{U}^*)\subseteq R(A_{V}^*)$. Next, we prove that $R(A_{V})\subseteq R(A_{U})$. We prove $N(A_{U}^*)\subseteq N(A_{V}^*)$, instead. Let $x\in N(A_{U}^*)$. Then $x\in N(B_{U}^*)$ as $N(A_{U}^*)\subseteq N(B_{U}^*)$. Set $z=\begin{pmatrix}
x\\
0
\end{pmatrix}$. We then have 
\begin{center}
$U^*z=\begin{pmatrix}
A_{U}^* & C_{U}^*\\
B_{U}^* & D_{U}^*
\end{pmatrix}
\begin{pmatrix}
x\\
0
\end{pmatrix}=
\begin{pmatrix}
A_{U}^*x\\
B_{U}^*x
\end{pmatrix}$.
\end{center}
Thus $z\in N(U^*)$ so that $z\in N(V^*)$. So, 
\begin{center}
$0=V^*z=\begin{pmatrix}
A_{V}^* & C_{V}^*\\
B_{V}^* & D_{V}^*
\end{pmatrix}
\begin{pmatrix}
x\\
0
\end{pmatrix}=
\begin{pmatrix}
A_{V}^*x\\
B_{V}^*x
\end{pmatrix}$.
\end{center}
So $x\in N(A_{V}^*)$ and hence $N(A_{U}^*)\subseteq N(A_{V}^*)$, proving that $(A_{U},A_{V})$ is a group inverse absorbing pair.
\end{proof}

The next result concerns the inheritance of the second principal diagonal subblock. The proof is similar to the proof of the Theorem $\ref{prinsub1g}$ and is omitted.

\begin{thm}\label{prinsub2g}
Let $U$ and $V$ be as above. Suppose that $R(C_{U})\subseteq R(D_{U})$ and $R(B_{V}^*)\subseteq R(D_{V}^*)$. If $(U,V)$ is a group inverse absorbing pair then $(D_{U},D_{V})$ is also a group inverse absorbing pair.
\end{thm}

In the concluding result, we prove the inheritance by the pseudo principal pivot transform.

\begin{thm}\label{inherppptg}
Let $U$ and $V$ be as above such that $A_{U}^{\#}$, $A_{V}^{\#}$, $D_{U}^{\#}$ and $D_{V}^{\#}$ exist. Suppose that $ R(C_{V}^*)\subseteq R(A_{U}^*)$, $R(B_{V})\subseteq R(A_{V})$, $R(C_{V})\subseteq R(D_{V})$, $R(B_{V}^*)\subseteq R(D_{V}^*)$, $R(C_{U}^*)\subseteq R(A_{U}^*)$, $R(B_{U})\subseteq R(A_{U})$, $R(C_{U})\subseteq R(D_{U})$ and $R(B_{U}^*)\subseteq R(D_{U}^*)$.\\
$(i)$ If $(U,V)$ is a group inverse absorbing pair then $(S_{U},S_{V})$ is also a group inverse absorbing pair, where 
\begin{center}
$S_{U}=pppt(U,A_{U})_{\#}=\begin{pmatrix}
A_{U}^\# & -A_{U}^\# B_{U}\\
C_{U}A_{U}^\# & K_{U}
\end{pmatrix}$ 
\end{center}
and 
\begin{center}
$S_{V}=pppt(V,A_{V})_{\#}=\begin{pmatrix}
A_{V}^\# & -A_{V}^\# B_{V}\\
C_{V}A_{V}^\# & K_{V}
\end{pmatrix}$.
\end{center}

$(ii)$ If $(U,V)$ is a group inverse absorbing pair then $(T_{U},T_{V})$ is also a group inverse absorbing pair, where 
\begin{center}
$T_{U}=pppt(U,D_{U})_{\#}=\begin{pmatrix}
L_{U} & B_{U}D_{U}^\# \\
-D_{U}^\# C_{U} & D_{U}^\#
\end{pmatrix}$
\end{center}
and 
\begin{center}
$T_{V}=pppt(V,D_{V})_{\#}=\begin{pmatrix}
L_{V} & B_{V}D_{V}^\# \\
-D_{V}^\# C_{V} & D_{V}^\#
\end{pmatrix}$.
\end{center}
\end{thm}
\begin{proof}
$(i)$ Since $(U,V)$ is a group inverse absorbing pair, from Theorem \ref{gpal}, we have $R(V)\subseteq R(U)$ and $R(U^*)\subseteq R(V^*)$. First, we show that $R(S_{V})\subseteq R(S_{U})$. Let $x=\begin{pmatrix}
x^{1}\\
x^{2}
\end{pmatrix}$ such that $x\in R(S_{V})$. The block matrices of $S_{V}$ satisfy the conditions in (Theorem $2.1$) \cite{kavi1}, so $S_{V}^{\#}$ exists and is given by 
\begin{center}
$S_{V}^{\#}=\begin{pmatrix}
L_{V} & B_{V}D_{V}^{\#}\\
-D_{V}^{\#}C_{V} & D_{V}^{\#}
\end{pmatrix}$. 
\end{center}
Now $R(S_{V})=R(S_{V}S_{V}^\#)$ and so we have $x \in R(S_{V}S_{V}^\#)$, where
\begin{center}
$S_{V}S_{V}^{\#}=\begin{pmatrix}
A_{V}^{\#}A_{V} & 0\\
0 & D_{V}D_{V}^{\#}
\end{pmatrix}$. 
\end{center}
Thus 
\begin{center}
$x^{1}\in R(A_{V}^\# A_{V})=R(A_{V}^\#)=R(A_{V})$ and $x^{2}\in R(D_{V}D_{V}^\#)=R(D_{V})$.
\end{center}
From Theorem \ref{prinsub1g} and Theorem \ref{prinsub2g}, both the pairs $(A_{U},A_{V})$ and $(D_{U},D_{V})$ are group inverse absorbing. So $R(A_{V})\subseteq R(A_{U})$ and $R(D_{V})\subseteq R(D_{U})$. Finally, we have 
\begin{center}
$x^{1}\in R(A_{U})=R(A_{U}A_{U}^\#)=R(A_{U}^\# A_{U})$
\end{center}
and 
\begin{center}
$x^{2}\in R(D_{U})=R(D_{U}D_{U}^\#)$.
\end{center}
Again, the block matrices of $S_{U}$ satisfy the conditions in (Theorem $2.1$) \cite{kavi1}. So, $S_{U}^{\#}$ exists and is given by 
\begin{center}
$S_{U}^{\#}=\begin{pmatrix}
L_{U} & B_{U}D_{U}^\#\\
-D_{U}^\# C_{U} & D_{U}^\#
\end{pmatrix}$.
\end{center}
Thus, we have 
\begin{center}
$S_{U}S_{U}^\# x=\begin{pmatrix}
A_{U}^\# A_{U}x^{1} & 0\\
0 & D_{U}D_{U}^\# x^{2}
\end{pmatrix}=\begin{pmatrix}
x^{1}\\
x^{2}
\end{pmatrix}=x$,
\end{center}
So, it follows that $x\in R(S_{U})$. Thus $R(S_{V})\subseteq R(S_{U})$. Let $y\in N(S_{V})$. Then $y\in N(S_{V}^{\#}S_{V})$, where 
\begin{center}
$S_{V}^{\#}S_{V}=\begin{pmatrix}
A_{V}A_{V}^{\#} & 0\\
0 & D_{V}^{\#}D_{V}
\end{pmatrix}$. 
\end{center}
\begin{center}
$y^{1}\in N(A_{V}A_{V}^{\#})=N(A_{V}^{\#}A_{V})=N(A_{V})$
\end{center}
and 
\begin{center}
$y^{2}\in N(D_{V}^{\#}D_{V})=N(D_{V})$.
\end{center}
Since $N(A_{V})\subseteq N(A_{U})$, we then have 
\begin{center}
$y^{1}\in N(A_{U})=N(A_{U}^{\#}A_{U})=N(A_{U}A_{U}^\#)$.
\end{center}
We also have, $N(D_{V})\subseteq N(D_{U})$ so that 
\begin{center}
$y^{2}\in N(D_{U})=N(D_{U}^\# D_{U})$. 
\end{center}
Hence, 
\begin{center}
$S_{U}^\# S_{U}y=\begin{pmatrix}
A_{U}A_{U}^{\#}y^{1} & 0\\
0 & D_{U}^{\#}D_{U}y^{2}
\end{pmatrix}=0$.
\end{center}
So $y\in N(S_{U}^\# S_{U})=N(S_{U})$, proving that  $R(S_{U}^*)\subseteq R(S_{V}^*)$. So $(S_{U},S_{V})$ is a group inverse absorbing pair. 

$(ii)$ Similar to part $(i)$.
\end{proof}


\newpage
\begin{thebibliography}{}
\bibitem{albert1969conditions}
A. Albert, {\it Conditions for positive and nonnegative definiteness in terms of pseudo inverses}, SIAM J. Appl. Math., {\bf17}, (1969) 434-440.
\bibitem{ben2003generalized}
A. Ben-Israel, T.N.E. Greville, {\it Generalized inverses: Theory and Applications}, second ed., Wiley, New York, 1974.
\bibitem{kavitaPseudoschur2015}
K. Bisht, G. Ravindran and K.C. Sivakumar, {\it Pseudo Schur complements, pseudo principal pivot transforms and their inheritance properties}, Electron. J. Linear Algebra, {\bf 30}, (2015) 455-477.
\bibitem{kavi1}
Kavita Bisht and K.C. Sivakumar, {\it Pseudo principal pivot transform: The group inverse case}, arXiv:1605.01970.
\bibitem{chen2008absorption}
X. Chen, J. Zhang, W. Guo, {\it The apsorption law of \{1,3\}, \{1,4\} inverse of matrix}, J. Inn. Mong. Norm. Univ., {\bf 37}, (2008) 337-339.
\bibitem{damm2009cancellation}
T. Damm and H. K. Wimmer, {\it A cancellation property of the Moore-Penrose inverse of triple products}, J. Aust. Math. Soc., {\bf 86}, (2009) 33-44.
\bibitem{greville1966note}
T.N.E. Greville, {\it Note on the generalized inverse of a matrix product}, SIAM Rev., {\bf 8}, (1966) 518-521.
\bibitem{hartwig1986reverse}
R.E. Hartwig, {\it The reverse order law revisited}, Linear Algebra Appl., {\bf 76}, (1986) 241-246.


\bibitem{lin2011mixedabsorption}
Y. Lin, Y. Gao, {\it Mixed absorption laws of the sum of two matrices on \{1,2\}-inverse and \{1,3\}-inverse}, J. Univ. Shanghai Sci. Technol., {\bf 33}, (2011) 168-173.
\bibitem{liu2012absorption}
X. Liu, H. Jin, D.S. Cvetkovi\'{c}-Ili\'{c}, {\it The absorption laws for the generalized inverses}, Appl. Math. and Computation, {\bf 219}, (2012) 2053-2059.
\bibitem{meenakshi1986principal}
 A.R. Meenakshi, {\it Principal pivot transforms of an EP matrix}, C.R. Math. Rep. Acad. Sci. Canada, {\bf 8}, (1986) 121-126.
\bibitem{wang2011associated}
 H. Wang and X. Liu, {\it The Associated Schur complements of $M=\begin{pmatrix}
A & B\\
C & D
\end{pmatrix}$}, Filomat, {\bf 25}, (2011) 155-161.

\end{thebibliography}
\end{document}